\documentclass{amsart}
\usepackage{amsbsy,amssymb,amscd,amsfonts,latexsym,amstext,delarray, amsmath,graphicx,color,caption,subcaption,enumerate,amsaddr,mathtools,mathabx}
\usepackage{graphicx}
\usepackage{epsfig}
\input xy

\xyoption{all}
\usepackage{euscript}
\usepackage{multirow}
\usepackage{etex, pictexwd,dcpic}
\usepackage[makeroom]{cancel}

\newtheorem{theorem}{Theorem}
\newtheorem{definition}[theorem]{Definition}
\newtheorem{proposition}[theorem]{Proposition}
\newtheorem{lemma}[theorem]{Lemma}
\newtheorem{corollary}[theorem]{Corollary}
\newtheorem{example}[theorem]{Example}
\newtheorem{conjecture}[theorem]{Conjecture}
\newtheorem{remark}[theorem]{Remark}

\newcommand{\db}[1]{\mathopen{\{\!\!\{}#1\mathopen{\}\!\!\}}}

\begin{document}
\title{Modified Double Poisson Brackets}
\author{S. Arthamonov}
\address{Department of Mathematics, Rutgers, The State University of New Jersey,\\
110 Frelinghuysen Rd., Piscataway, NJ 08854, USA.
}
\email{semeon.artamonov@rutgers.edu}
\date{}
\begin{abstract}
We propose a non-skew-symmetric generalization of the original definition of double Poisson Bracket by M. Van den Bergh. It allows one to explicitly construct more general class of $H_0$-Poisson structures on finitely generated associative algebras. We show that modified double Poisson brackets inherit certain major properties of the double Poisson brackets.
\end{abstract}
\maketitle

\section{Introduction}

Application of the noncommutative geometry program to symplectic manifolds was originated in \cite{Kontsevich'1993}. Following general philosophy formulated by M.~Kontsevich any algebraic property that makes geometric sense is mapped to its commutative counterpart by the functor $\mathrm{Rep}_N$:
\begin{align*}
\mathrm{Rep}_N:\quad\mathrm{fin.\;gen.\;Associative\;algebras} \rightarrow\mathrm{Affine\; schemes},
\end{align*}
which assigns to a finitely generated associative algebra $\mathcal A$ a scheme of its' $N\times N$ matrix representations\footnote{Throughout the text we assume the ground field to be $\mathbb C$. All unadorned tensor products, algebras, and schemes are over $\mathbb C$ unless specified otherwise.}
\begin{align*}
\qquad\mathrm{Rep}_N(\mathcal A)=Hom(\mathcal A,Mat_N(\mathbb C)).
\end{align*}
In line with this general philosophy, M.~Van~den~Bergh \cite{VandenBergh'2008} proposed a definition of the double Poisson bracket on associative algebra which induces a conventional Poisson bracket on the coordinate ring of matrix representations.
%\begin{definition}\cite{VandenBergh'2008} A modified double Poisson bracket on associative algebra $\mathcal A$ is a linear map $\db{,}:\;\mathcal A\otimes\mathcal A\rightarrow\mathcal A\otimes\mathcal A$ s.t. for all $a,b,c\in\mathcal A$
%\begin{enumerate}
%\item $\db{a,b}=-\db{b,a}^{op},$
%\item $\db{ab,c}=(1\otimes a)\db{b,c}+\db{a,c}(b\otimes1),$
%\item $\db{a,bc}=(b\otimes1)\db{a,c}+\db{a,b}(1\otimes c),$
%\item $R_{12}R_{23}+R_{31}R_{12}+R_{23}R_{31}=0,$ where $R_{m,n}:\;\underbrace{\mathcal A\otimes\mathcal  A\otimes\dots\otimes\mathcal  A}_{k\;\textrm{times}}\rightarrow\underbrace{\mathcal A\otimes\mathcal  A\otimes\dots\otimes\mathcal  A}_{k\;\textrm{times}}$ \begin{align*}
%    R_{m,n}(a_1\otimes\dots\otimes a_k)=a_1\otimes\dots\otimes a_{m-1}\otimes\db{a_m,a_n}'\otimes\dots\otimes
%    \db{a_m,a_n}''\otimes\dots\otimes a_k.
%    \end{align*}
%\end{enumerate}
%\end{definition}

On the contrary, W.~Crawley-Boevey \cite{Crawley-Boevey'2011,Crawley-BoeveyEtingofGinzburg'2007} suggested related, yet another definition of the noncommutative analogue of the Poisson bracket, the so-called $H_0$-Poisson structure. The latter has weaker requirements and in general provides a conventional Poisson bracket only on the moduli space of representations. Double Poisson bracket induces an $H_0$-Poisson structure but not vice versa.

One of the major advantages of the double Poisson bracket as opposed to an $H_0$-Poisson structure is that for a finitely generated associative algebra it is defined completely by its' action on generators. This allows one to provide numerous explicit examples of double Poisson brackets \cite{PichereauVandeWeyer'2008,BartocciTacchella'2016} and even carry out certain partial classification problems \cite{OdesskiiRubtsovSokolov'2013}.

In this note we provide an extension of the original ideas of M.~Van~den~Bergh and W.~Crawley-Boevey. We define a notion of the \textit{modified} double Poisson bracket (see Definition \ref{def:ModifiedDoubleBracket}) which allows one to construct explicitly more general examples of $H_0$-Poisson brackets on finitely generated algebras. We support our definition with new examples of modified double Poisson brackets in Sec. \ref{sec:ExamplesOfModifiedDoublePoissonBrackets} and calculate corresponding dimensions of symplectic leaves of the induced Poisson structures on the moduli space.

In Section \ref{sec:GeneralPoissonPolyvectors} we use the algebra of noncommutative poly-vector fields introduced in \cite{VandenBergh'2008} to construct non-skew-symmetric biderivations. We introduce the notion of a modified double Poisson bivector and present an essential example.

Finally, in Section \ref{sec:RepresentationAlgebras} we investigate brackets on representation algebras induced by the modified double Poisson brackets. We show that some recent results of G.~Massuyeau and V.~Turaev \cite{MassuyeauTuraev'2015} can be extended beyond skew-symmetric case as well.

\section{Modified double Poisson bracket}

Let $\mathcal A=\mathbb C\langle x^{(1)},\dots,x^{(k)}\rangle/\mathcal R$ be an associative algebra over $\mathbb C$, which is finitely generated by $\{x^{(1)},\dots x^{(k)}\}$, possibly with some finite number of relations $\mathcal R$.
\begin{definition}
\label{def:ModifiedDoubleBracket}
%\footnote{I've added requirement of weak skew symmetry to simplify the proof of Proposition \ref{prop:InvariantPoisson}, although I'm still unsure whether it is independent from the previous 3 requirements. Easy to see that from Jacobi identity one immediately gets that $\{a,b\}+\{b,a\}\in I$, where $\{I,A\}=0$. Actually, (1)-(3) put more severe restrictions on $I$ when we consider monomials of higher order. Finally, in the classification problem for quadratic modified double Poisson brackets (4) is satisfied automatically by all cases.}
A modified double Poisson bracket on $\mathcal A$ is a map $\mathcal A\otimes\mathcal A\rightarrow\mathcal A\otimes\mathcal A$ s.t. for all $a,b,c\in\mathcal A$
\begin{subequations}
\begin{align}
&\db{a\otimes bc}=(b\otimes 1)\db{a\otimes c}+\db{a\otimes b}(1\otimes c)
\label{eq:ModifiedDoublePoissonLeibnitz1}\\
&\db{ab\otimes c}=(1\otimes a)\db{b\otimes c}+\db{a\otimes c}(b\otimes 1)
\label{eq:ModifiedDoublePoissonLeibnitz2}\\
&\{a\otimes\{b\otimes c\}\}-\{b\otimes\{a\otimes c\}\}=\{\{a\otimes b\}\otimes c\}\quad\mathrm{where}\quad\{\_\}:=\mu\circ \db{\_}
\label{eq:ModifiedDoublePoissonJacobi}\\
&\{a,b\}+\{b,a\}=0\,\bmod [\mathcal A,\mathcal A]
\label{eq:ModifiedDoublePoissonWeakSkewSymmetry}
\end{align}
\end{subequations}
\end{definition}

The fact that we do not require skew-symmetry in a sense of Van den Bergh $\db{a,b}=-\db{b,a}^{op}$ is the major difference with the case studied in \cite{VandenBergh'2008,OdesskiiRubtsovSokolov'2012}. To distinguish with definitions introduced in \cite{VandenBergh'2008} we call this object \textbf{Modified} double Poisson Bracket. In Section \ref{sec:ExamplesOfModifiedDoublePoissonBrackets} we show that there are essential examples of the modified double Poisson brackets.

\begin{corollary}
Composition with the multiplication map $\{\_\}:\mathcal A\otimes\mathcal A\rightarrow\mathcal A$ defines an $H_0$-Poisson structure, namely for all $a,b,c\in\mathcal A$
\begin{subequations}
\begin{align}
&\{a,bc\}=b\{a,c\}+\{a,b\}c,\label{eq:H0Derivation}\\
&\{ab,c\}=\{ba,c\},\label{eq:HOCyclicInvariance}\\
&\{a,\{b,c\}\}-\{b,\{a,c\}\}=\{\{a,b\},c\}, \label{eq:H0LeftLodayJacobi}\\
&\{a,b\}+\{b,a\}\equiv0\bmod [\mathcal A,\mathcal A].
\end{align}
\end{subequations}
\label{cor:H0FromModifiedDoublePoisson}
\end{corollary}
In particular, the latter implies that $\{\_\}:\mathcal A\otimes\mathcal A\rightarrow\mathcal A$ factors through $\{\_\}:\mathcal A/[\mathcal A,\mathcal A]\otimes\mathcal A\rightarrow\mathcal A$ which we denote by the same brackets.

\begin{corollary}
An $H_0$-Poisson structure $\{\_\}$ in turn induces a Lie Algebra structure $\{\_\}^{Lie}:\mathcal A/[\mathcal A,\mathcal A]\otimes \mathcal A/[\mathcal A,\mathcal A]\rightarrow\mathcal A/[\mathcal A,\mathcal A]$ on abelianization $\mathcal A_\natural=\mathcal A/[\mathcal A,\mathcal A]$ of $\mathcal A$.
\end{corollary}

\section{Poisson brackets on the moduli space of representations}

\label{sec:BracketsOnMatNC}

Double derivation property introduced to Definition \ref{def:ModifiedDoubleBracket} at the same time provides a constructive definition for a certain subclass of $H_0$-Poisson structures and allows one to establish a precise correspondence between $H_0$-Poisson structures and geometry. Throughout this section we will review main ideas of pioneering papers \cite{Crawley-Boevey'1999, Crawley-BoeveyEtingofGinzburg'2007, VandenBergh'2008} and apply them to the context of the Modified Double Poisson Bracket.

\subsection{Representation scheme}

As before, let $\mathcal A=\langle x^{(1)},\dots,x^{(k)}\rangle/\mathcal R$ be a finitely generated associative algebra with a finite set of relations $\mathcal R$. Each representation of $\mathcal A$ in $Mat_N(\mathbb C)$ can be defined by the image of the generators, let
\begin{align}
\varphi(x^{(i)})=\left(\begin{array}{ccc}
x^{(i)}_{11}&\dots&x^{(i)}_{1N}\\
\vdots&&\vdots\\
x^{(i)}_{N1}&\dots&x^{(i)}_{NN}
\end{array}\right).
\label{eq:NXNRepresentationDefinition}
\end{align}
Representations of $\mathcal A$ then form an affine scheme $\mathcal V$  with a coordinate ring $\mathbb C[\mathcal V]:=\mathbb C\left[x^{(i)}_{j,k}\right
%,\,1\leqslant i\leqslant k,\,1\leqslant j,k\leqslant N
]/\varphi(\mathcal R)$. Denote as $\mathbb C_{\mathcal V}$ --- the corresponding sheaf of rational functions. Then $\varphi:\,\mathcal V\times\mathcal A\rightarrow Mat_N(\mathbb C)$. For a general point $m\in\mathcal V$ map $\varphi(m,\_\,)$ provides an $N$-dimensional matrix representation of $\mathcal A$. Hereinafter, we often omit the first argument of $\varphi$ where it is assumed to be a function on $\mathcal V$.

\subsection{Moduli space of representations}

There is a natural action of $GL_N(\mathbb C)\circlearrowleft Mat_N(\mathbb C)$ which corresponds to the change of basis in the underlying finite dimensional module. It induces the $GL_N(\mathbb C)$ action on the sheaf of rational functions $\mathbb C_{\mathcal V}$. We denote as $\mathbb C[\mathcal V]^{inv}\subset\mathbb C[\mathcal V]$ (respectively $\mathbb C_{\mathcal V}^{inv}\subset\mathbb C_{\mathcal V}$) the subalgebra of $GL_N(\mathbb C)$ invariant elements. We refer to the orbit of the $GL_N(\mathbb C)$ action as an isomorphism class of representations and thus $\mathbb C[\mathcal V]^{inv}$ is the coordinate ring of the corresponding moduli space.

One can construct elements of $\mathbb C[\mathcal V]^{inv}$ by taking traces $\varphi_{ii}(x)$ for different $x\in A$, clearly the image would be invariant under the cyclic permutations of generators in each monomial and thus would depend only on the element of the cyclic space $A_\natural=\mathcal A/[\mathcal A,\mathcal A]$. This induces a map $\varphi_0:\mathcal A/[\mathcal A,\mathcal A]\rightarrow\mathbb C[\mathcal V]^{inv}$ from $\mathcal A/[\mathcal A,\mathcal A]$ to the invariant subalgebra $\mathbb C[\mathcal V]^{inv}$. Denote the image of this map by $\mathcal H:=\varphi_0(\mathcal A/[\mathcal A,\mathcal A])$.
\begin{lemma}{\cite{Procesi'1976}}
$C[\mathcal V]^{inv}$ is generated by $\mathcal H$ as a commutative algebra.
\label{lemm:Processi}
\end{lemma}

\begin{example}
If $A$ happens to be commutative, the representation functor $\mathrm{Rep}_1$ for $N=1$ will map it to itself. Moreover $\mathbb C_{\mathcal V}^{inv}$ will coincide with $\mathbb C_{\mathcal V}$ for this case.
\label{rem:Abelian}
\end{example}

\begin{example}
The simplest scheme here corresponds to the representations of $A=\mathbb C\langle x^{(1)},\dots,x^{(k)}\rangle$ --- free algebra with $k$ generators. In the absence of relations, the corresponding scheme is birational to $\mathbb C^{kN^2}$, a $kN^2$-dimensional vector space. And the corresponding sheaf of rational functions is nothing but the field of rational functions in $kN^2$ variables.
\end{example}

\begin{example}
Another interesting special case corresponds to the so-called smooth algebras. Finitely generated algebra $A$ is called smooth if\; $\Omega^1:=\ker\mu=\{a_1\otimes a_2\,| a_1a_2=0,\;a_1,a_2\in A\}$ is projective as an inner bimodule. This guarantees that the representation scheme is actually a smooth affine variety. This case was in details studied in \cite{Crawley-BoeveyEtingofGinzburg'2007}.
\end{example}

The major advantage of the Poisson formalism as compared to the Symplectic formalism is that it can be easily generalized beyond the smooth case.

\subsection{Bracket}

Define induced bracket $\{,\}^{\mathcal V}:\mathbb C_{\mathcal V}\otimes\mathbb C_{\mathcal V}\rightarrow\mathbb C_{\mathcal V}$ on generators  $x^{(m)}_{ij}$ of $\mathbb C[\mathcal V]$ by
\begin{subequations}
\begin{align}
\left\{x^{(m)}_{ij},x^{(n)}_{kl}\right\}^{\mathcal V}=\varphi\left(\db{x^{(m)}\otimes x^{(n)}}\right)_{(kj),(il)}
\label{eq:BracketOnCoordinateSpace}
\end{align}
And then extend it to the entire $\mathbb C_{\mathcal V}$ by Leibnitz identities w.r.t. both arguments. Namely, for all $a,b,c\in\mathbb C_{\mathcal V}$.
\begin{align}
\{ab,c\}^{\mathcal V}=&a\{b,c\}^{\mathcal V}+b\{a,c\}^{\mathcal V},
\label{eq:BracketOnCoordinateSpaceLeftLeibnitzIdentity}\\
\{a,bc\}^{\mathcal V}=&c\{a,b\}^{\mathcal V}+b\{a,c\}^{\mathcal V}.
\label{eq:BracketOnCoordinateSpaceRightLeibnitzIdentity}
\end{align}
\label{eq:BracketNDefinition}
\end{subequations}

As opposed to \cite{VandenBergh'2008}, bracket (\ref{eq:BracketNDefinition}) in the context of Definition \ref{def:ModifiedDoubleBracket} is not necessary skew-symmetric and thus is not yet a Poisson bracket on $\mathbb C_{\mathcal V}$. It is a famous result of W.~Crawley-Boevey \cite{Crawley-Boevey'2011} that any $H_0$-Poisson structure induces a conventional Poisson bracket on the moduli space of representations. In addition to that, we show in Proposition \ref{prop:H0BracketNXN} that it comes with a Lie module action on the coordinate space of representations. In the case of bracket induced by the modified double Poisson bracket both are nothing but restrictions of (\ref{eq:BracketNDefinition}) to $\mathbb C_{\mathcal V}^{inv}\otimes\mathbb C_{\mathcal V}^{inv}$ and $\mathbb C_{\mathcal V}^{inv}\otimes\mathbb C_{\mathcal V}$ respectively.

Easy to check that the above extension (\ref{eq:BracketOnCoordinateSpaceLeftLeibnitzIdentity}) -- (\ref{eq:BracketOnCoordinateSpaceRightLeibnitzIdentity}) is consistent with the double Leibnitz identity and relations $\varphi(\mathcal R)$ in the coordinate ring $\mathbb C[\mathcal V]$, namely
\begin{lemma}
Equations (\ref{eq:BracketNDefinition}) define a unique linear map $\{,\}^{\mathcal V}:\mathbb C[\mathcal V]\otimes\mathbb C[\mathcal V]\rightarrow\mathbb C[\mathcal V]$ given by
\label{lemm:BracketCoordinateRing}
\begin{align}
\forall x,y\in A:\,\{\varphi(x)_{ij},\varphi(y)_{kl}\}^{\mathcal V}= \varphi(\db{x,y}')_{kj}\varphi(\db{x,y}'')_{il}
\label{eq:NXNDoubleBracket}
\end{align}
\label{lemm:NXNDoubleBracket}
\end{lemma}
\begin{proof}
Define $X=\varphi(x)$ and $Y=\varphi(y)$. In what follows assume the summation over repeating indexes
\begin{align*}
\{X_{ij},Y_{kl}Z_{lm}\}^{\mathcal V}=\varphi(\db{x,yz}')_{kj}\varphi(\db{x,yz}'')_{im}.
\end{align*}
On the other hand
\begin{align*}
\db{x,yz}=&(y\otimes 1)\db{x,z}+\db{x,y}(1\otimes z)\\
=&y\db{x,z}'\otimes\db{x,z}''+\db{x,y}'\otimes\db{x,y}''z
\end{align*}
Which leads us to
\begin{align*}
\{X_{ij},Y_{kl}Z_{lm}\}^{\mathcal V}=&Y_{kl}\varphi(\db{x,z}')_{lj}\varphi(\db{x,z}'')_{im} +\varphi(\db{x,y}')_{kj}\varphi(\db{x,y}'')_{il}Z_{lm}\\
=&Y_{kl}\{X_{ij},Z_{lm}\}^{\mathcal V}+\{X_{ij},Y_{kl}\}^{\mathcal V}Z_{lm}
\end{align*}
By the same derivation
\begin{align*}
\{X_{il}Y_{lj},Z_{km}\}^{\mathcal V}=X_{il}\{Y_{lj},Z_{km}\}^{\mathcal V}+Y_{lj} \{X_{il},Z_{km}\}^{\mathcal V}.
\end{align*}
Now, using the fact that $\mathcal A$ is finitely generated by $x^{(m)}$ we conclude that $(\ref{eq:BracketOnCoordinateSpaceLeftLeibnitzIdentity})$ and $(\ref{eq:BracketOnCoordinateSpaceRightLeibnitzIdentity})$ uniquely extend $\{\_,\_\}^{\mathcal V}$ on pairs of monomials. Moreover, defining ideal $\varphi(\mathcal R)\subset\mathbb C[\mathcal V]$ for the coordinate ring $\mathbb C[\mathcal V]$ is thus within the left and right kernel of $\{\_,\_\}^{\mathcal V}$. So $\{\_,\_\}^{\mathcal V}$ extends uniquely to $\mathbb C[\mathcal V]$.
\end{proof}

Equation (\ref{eq:NXNDoubleBracket}) immediately implies
\begin{corollary}
For all $a,b\in A,$ $\{\varphi_0(a),\varphi(b)\}^{\mathcal V}=\varphi(\{a,b\}).$
\label{cor:DerivationHomomorphism}
\end{corollary}
\begin{proof}
\begin{align}
\forall x,y\in A:\,\{\varphi(x)_{ii},\varphi(y)_{kl}\}^{\mathcal V}=& \varphi(\db{x,y}')_{ki}\varphi(\db{x,y}'')_{il}= \varphi(\db{x,y}'\db{x,y}'')_{kl}=\varphi(\mu(\db{x,y}))_{kl}=\nonumber\\ =&\varphi(\{x,y\}^{\mathcal V})_{kl}.
\label{eq:NXNLodayBracket}
\end{align}
\end{proof}

\begin{lemma}
Equations (\ref{eq:BracketNDefinition}) define a unique linear map $\{,\}^{\mathcal V}:\;\mathbb C_{\mathcal V}\otimes\mathbb C_{\mathcal V}\rightarrow\mathbb C_{\mathcal V}$.
\end{lemma}

\begin{proof}
Taking into account Lemma \ref{lemm:NXNDoubleBracket} it would be enough to prove that the bracket $\{,\}^{\mathcal V}$ can be extended to a properly localized ring. Let $R$ be a $\mathbb C$-algebra s.t. $\{,\}^{\mathcal V}:R\otimes R\rightarrow R$ is well defined and satisfies (\ref{eq:BracketOnCoordinateSpaceLeftLeibnitzIdentity}) -- (\ref{eq:BracketOnCoordinateSpaceRightLeibnitzIdentity}). For any multiplicative subset $S\subset R$ and $a\in S,\;b\in R$ we immediately get  $\{a^{-1},b\}^{\mathcal V}=-a^{-2}\{a,b\}^{\mathcal V}$ and $\{b,a^{-1}\}^{\mathcal V}=-a^{-2}\{b,a\}^{\mathcal V}$. This provides a unique extension of $\{,\}^{\mathcal V}$ to $S^{-1}R$.

Thus for each distinguished open subset $\mathcal V_f\subset\mathcal V$ we have a unique extension of $\{,\}^{\mathcal V}$ to $\Gamma(\mathcal V_f,{\mathcal O}_{\mathcal V})$. Now denote by $S(\mathcal V_f)\subset\Gamma(X_f,{\mathcal O}_{\mathcal V})$ the set of functions which are not a zero divisor on any stalk, we have a unique extension of $\{,\}^{\mathcal V}$ to $S(\mathcal V_f)^{-1}\Gamma(\mathcal V_f,{\mathcal O}_{\mathcal V})=\Gamma(\mathcal V_f,\mathbb C_{\mathcal V})$.
\end{proof}

\begin{lemma}
Following restriction of $\{,\}^{\mathcal V}$
\begin{align*}
\{,\}^{inv}:\mathbb C_{\mathcal V}^{inv}\otimes\mathbb C_{\mathcal V}^{inv}\rightarrow\mathbb C_{\mathcal V}^{inv}
\end{align*}
is skew-symmetric, namely for all $f,g\in\mathbb C_{\mathcal V}^{inv}$ we have $\{f,g\}^{\mathcal V}\in\mathbb C_{\mathcal V}^{inv}$ and $\{f,g\}^{\mathcal V}=-\{g,f\}^{\mathcal V}$.
\label{lemm:BracketNSkewSymmetricRestriction}
\end{lemma}
\begin{proof}
In light of Leibnitz identities (\ref{eq:BracketOnCoordinateSpaceLeftLeibnitzIdentity}) and (\ref{eq:BracketOnCoordinateSpaceRightLeibnitzIdentity}) it would be enough for us to show the statement for generators of $\mathbb C[\mathcal V]^{inv}$. So w.l.o.g we can assume that $f,g\in\mathcal H$ (see Lemma \ref{lemm:Processi}). Under this assumption there exist $x,y\in A$ s.t. $f=\varphi_0(x)$ and $g=\varphi_0(y)$. Denote $X:=\varphi(x),\; Y:=\varphi(y)$. Using (\ref{eq:NXNLodayBracket}) we get
\begin{align*}
\{f,g\}^{\mathcal V}=\{X_{ii},Y_{kk}\}^{\mathcal V}=\varphi(\{x,y\})_{kk} =\varphi_0(\{x,y\})\;\in\mathbb C[\mathcal V]^{inv}
\end{align*}
as a result
\begin{align*}
\{f,g\}^{\mathcal V}+\{g,f\}^{\mathcal V}=\varphi_0(\{x,y\}+\{y,x\})=0.
\end{align*}
\end{proof}

\begin{proposition}
The following restriction
\begin{align}
\{\_,\_\}^{\mathcal V}:\quad \mathbb C_{\mathcal V}^{inv}\otimes\mathbb C_{\mathcal V}\rightarrow\mathbb C_{\mathcal V}
\label{eq:LodayNXNDefinition}
\end{align}
satisfies the Jacobi identity for the left Loday bracket, for all $f,g\in\mathbb C_{\mathcal V}^{inv}$ and $h\in\mathbb C_{\mathcal V}:$
\begin{align*}
\{f,\{g,h\}^{\mathcal V}\}^{\mathcal V}-\{g,\{f,h\}^{\mathcal V}\}^{\mathcal V}=\{\{f,g\}^{\mathcal V},h\}^{\mathcal V}.
\end{align*}
%and for each $f\in\mathbb C_{\mathcal V}^{inv}$ defines a derivation of $\mathbb C_{\mathcal V}$, namely for all $g,h\in\mathbb C_{\mathcal V}$
%\begin{align*}
%\{f,gh\}^{\mathcal V}=g\{f,h\}^{\mathcal V}+\{f,g\}^{\mathcal V}h.
%\end{align*}
\label{prop:H0BracketNXN}
\end{proposition}
\begin{proof}
For $f,g\in\mathcal H$ and $h\in\mathbb C[\mathcal V]$ the statement is a straightforward consequence of Corollary \ref{cor:DerivationHomomorphism} and the fact that $\{,\}:\mathcal A_\natural\otimes\mathcal A\rightarrow\mathcal A$ is a Loday bracket. Denote
\begin{align*}
\phi(f,g,h):=\{f,\{g,h\}^{\mathcal V}\}^{\mathcal V}-\{g,\{f,h\}^{\mathcal V}\}^{\mathcal V}-\{\{f,g\}^{\mathcal V},h\}^{\mathcal V}.
\end{align*}
Since $\phi$ is a derivation in its' last argument, left Jacobi identity extends for $h\in\mathbb C_{\mathcal V}$. Next, we have
\begin{align*}
\phi(f_1f_2,g,h)-&f_1\phi(f_2,g,h)-f_2\phi(f_1,g,h)=\\
=& -\{g,f_1\}^{\mathcal V}\{f_2,h\}^{\mathcal V}- \{g,f_2\}^{\mathcal V}\{f_1,h\}^{\mathcal V}- \{f_2,g\}^{\mathcal V}\{f_1,h\}^{\mathcal V}- \{f_2,h\}^{\mathcal V}\{f_1,g\}^{\mathcal V}
\end{align*}
Note, that $\{,\}^{\mathcal V}$ is not skew-symmetric in general, however by Lemma \ref{lemm:BracketNSkewSymmetricRestriction} we have for $f_1,f_2,g\in\mathbb C_{\mathcal V}^{inv}$
\begin{align*}
\{g,f_1\}^{\mathcal V}+\{f_1,g\}^{\mathcal V}=0,\qquad \{f_2,g\}^{\mathcal V}+\{g,f_2\}^{\mathcal V}=0
\end{align*}
which is enough to conclude that for all $f_1,f_2,g\in\mathbb C_{\mathcal V}^{inv}$ and $h\in\mathbb C_{\mathcal V}$
\begin{align*}
\phi(f_1f_2,g,h)=f_1\phi(f_2,g,h)+f_2\phi(f_1,g,h).
\end{align*}
We also get for all $f,g\in\mathbb C_{\mathcal V}^{inv}$ and $g\in\mathbb C_{\mathcal V}$ s.t. $f^{-1}\in\mathbb C_{\mathcal V}^{inv}$
\begin{align*}
\phi(f^{-1},g,h)=-f^{-2}\phi(f,g,h).
\end{align*}
Similar reasoning applies for the second argument which finalizes the proof.
\end{proof}

\begin{remark}
Proposition \ref{prop:H0BracketNXN} defines a representation analogue of an $H_0$-Poisson structure. Note the dual properties, once $H_0$-Poisson structure factors through $\{,\}:\;\mathcal A/[\mathcal A,\mathcal A]\otimes\mathcal A\rightarrow\mathcal A$, the induced bracket defined above has to be restricted on invariant subalgebra $\{,\}^{\mathcal V}:\;\mathbb C_{\mathcal V}^{inv}\otimes\mathbb C_{\mathcal V}\rightarrow\mathbb C_{\mathcal V}$ in order to satisfy Jacobi identity. Following ideas of \cite{Turaev'2014} we formulate this duality fundamentally in Section \ref{sec:BracketsOnRepresentationAlgebras} when we show that one can generalize Proposition \ref{prop:H0BracketNXN} beyond matrix representations.
\end{remark}

\begin{corollary} The following restriction
\begin{align}
\{\_,\_\}^{inv}:\quad\mathbb C_{\mathcal V}^{inv}\otimes\mathbb C_{\mathcal V}^{inv}\rightarrow\mathbb C_{\mathcal V}^{inv}
\label{eq:InvariantPoissonBracket}
\end{align}
is a Poisson bracket.
\label{prop:InvariantPoisson}
\end{corollary}
\begin{proof}
This statement follows from Corollary \ref{cor:H0FromModifiedDoublePoisson} and results of \cite{Crawley-Boevey'2011}, however below we will present a direct proof using Proposition \ref{prop:H0BracketNXN}.

Indeed, by Lemma \ref{lemm:BracketNSkewSymmetricRestriction} this restriction is skew-symmetric, it satisfies Leibnitz identity in both arguments by definition (\ref{eq:BracketOnCoordinateSpaceLeftLeibnitzIdentity}) -- (\ref{eq:BracketOnCoordinateSpaceRightLeibnitzIdentity}). As a particular case of Proposition \ref{prop:H0BracketNXN} it also satisfies Jacobi identity.
\end{proof}

\subsection{Casimir elements}

Action of the Poisson bracket (\ref{eq:InvariantPoissonBracket}) on the full representation scheme (\ref{eq:LodayNXNDefinition}) provides a convenient way to construct Casimir elements. Recall

\begin{definition}
Element $c\in\mathcal A$ is a right Casimir of bracket $\{,\}$ if for all $h\in\mathcal A/[\mathcal A,\mathcal A]$ we have $\{h,c\}=0$.
\end{definition}

\begin{remark}
It is worth noting that right Casimir elements are not necessary within the left kernel of the bracket beyond skew-symmetric case. For a particular example see (\ref{eq:DoubleBracketForKontsevichSystem}) and (\ref{eq:TypeARightCasimir}).
\end{remark}

Since $\{,\}$ is a derivation in the second argument, the set of all right Casimir elements forms a subalgebra $\mathcal C\subset\mathcal A$. This subalgebra allows one to construct Casimir elements of the bracket on representation scheme.

\begin{proposition}
Assume that $c\in A$ is a right Casimir of bracket $\{,\}:A/[A,A]\otimes A\rightarrow A$. Subalgebra $\mathbb C[\varphi(c)_{kl}]$ generated by components of the matrix $\varphi(c)$ consist of right Casimirs of the bracket $\{,\}^{\mathcal V}$
\label{prop:CasimirsNXN}
\end{proposition}
\begin{proof}
For all $h\in\mathcal H$ we have
\begin{align*}
\{\varphi_0(h),\varphi(c)_{kl}\}^{\mathcal V}=\{\varphi(h)_{ii},\varphi(c)_{kl}\}^{\mathcal V}=\varphi(\{h,c\})_{kl}=0.
\end{align*}
Since $\{,\}^{\mathcal V}$ satisfies Leibnitz identity w.r.t to both arguments we conclude that for all $f\in\mathbb C_{\mathcal V}^{inv}$ and $x\in\mathbb C[\varphi(c)_{kl}]$
\begin{align*}
\{f,x\}^{\mathcal V}=0.
\end{align*}
\end{proof}
Note, that in Proposition \ref{prop:CasimirsNXN} it is essential to consider a restricted bracket $\{\_,\_\}^{\mathcal V}$ on $\mathbb C_{\mathcal V}^{inv}\otimes\mathbb C_{\mathcal V}$ as defined in (\ref{prop:CasimirsNXN}). If instead of element $\mathbb C_{\mathcal V}^{inv}$ as a first argument we take arbitrary $f\in\mathbb C_{\mathcal V}$ the bracket with a Casimir elements do not have to be zero.
\begin{corollary}
Let $\mathcal C\subset\mathcal A$ be a subalgebra of right Casimir elements of bracket $\{,\}$, then $\varphi_0(\mathcal C)\subset\mathbb C_{\mathcal V}^{inv}$ consist of Casimir elements of bracket $\{,\}^{inv}$.
\end{corollary}

This Corollary is especially useful when $\mathcal C$ is finitely generated. We illustrate this method in Section \ref{sec:DimensionsOfSymplecticLeavesForKontsevichBracket}.

\section{Examples of the Modified Double Poisson brackets}
\label{sec:ExamplesOfModifiedDoublePoissonBrackets}

\subsection{Bracket for Kontsevich system}
\label{sec:BracketForKontsevichSystem}

Here we describe a particular example of modified double bracket on $\mathbb C\langle u^{\pm},v^{\pm}\rangle$ introduced in \cite{Arthamonov'2015}. This double bracket is not skew-symmetric and thus provides an example beyond the case considered in \cite{VandenBergh'2008}.

Let $\mathcal A^+=\mathbb C\langle u,v\rangle$ be a free associative algebra with two generators. Define a biderivation of $\mathcal A$ on the generators as
\begin{align}
\db{u,v}_K=-vu\otimes1,\qquad
\db{v,u}_K=uv\otimes 1,\qquad \db{u,u}_K=\db{v,v}_K=0.
\label{eq:DoubleBracketForKontsevichSystem}
\end{align}

\begin{proposition}{\cite{Arthamonov'2015}}
Biderivation $\db{\_}_K$ is a modified double Poisson bracket.
\end{proposition}

Under the representation functor $\mathrm{Rep_N}$ our algebra is mapped to the commutative algebra $\mathcal A_N=\mathbb C(u_{i,j},v_{i,j})$ of rational functions in $2N^2$ variables $\{u_{i,j},v_{i,j}\,|\,1\leqslant i,j\leqslant N\}$. The corresponding affine scheme $\mathcal V$ is just a $2N^2$-dimensional vector space over $\mathbb C$.

The induced bracket is a biderivation $\{,\}^{\mathcal V}:\mathcal A_N\otimes\mathcal A_N\rightarrow\mathcal A_N$ defined on generators as
\begin{equation}
\begin{aligned}
\{u_{ij},v_{kl}\}^{\mathcal V}=&-\delta_{il}\sum_mv_{km}u_{mj}\\
\{v_{kl},u_{ij}\}^{\mathcal V}=&\delta_{kj}\sum_mu_{im}v_{ml}.
\end{aligned}
\label{eq:KontevichBracketNxN}
\end{equation}

Proposition \ref{prop:InvariantPoisson} implies that restriction of $\{,\}^{\mathcal V}$ on invariant rational functions $\{\_,\_\}^{inv}:\,\mathcal A_N^{inv}\otimes\mathcal A_N^{inv}\rightarrow\mathcal A_N^{inv}$ is a Poisson bracket.

\subsubsection{Dimensions of the symplectic leaves}
\label{sec:DimensionsOfSymplecticLeavesForKontsevichBracket}

Element
\begin{align}
c=uvu^{-1}v^{-1}
\label{eq:TypeARightCasimir}
\end{align}
is a right Casimir of the $H_0$-Poisson bracket induced by (\ref{eq:DoubleBracketForKontsevichSystem}) (see \cite{Arthamonov'2015} for a proof). One can show that $\textrm{Tr}\,\varphi(c^k)$ provide Casimirs for the indeuced bracket $\{\_,\_\}^{inv}$. The Poisson bracket on $\mathbb C^{inv}_{\mathcal V}$ we defined earlier is degenerate due to existence of Casimirs. Which means that the Poisson tensor is not invertible at a generic point. In order to make it invertible (and thus induce a symplectic structure) one has to restrict the bracket to the subvariety corresponding to the fixed level of all Casimir functions (See e.g. \cite{Arnold}). The codimension of such variety is, of course, simply the number of algebraically independent Casimir functions.

Based on direct computation of dimensions of symplectic leaves for bracket $\{\_,\_\}^{inv}$ we come to the following
\begin{conjecture}
There are exactly $N-1$ algebraically independent Casimir elements given by $\textrm{Tr}\,\varphi(c^k)$ for the bracket $\{,\}^{inv}$.
\end{conjecture}
We summarize a computational evidence in favour of this conjecture in the Table \ref{tab:QuadraticBracketSymplecticLeafDimensions}. Here $\textrm{dim}\,L$ --- dimension of a generic symplectic leaf, $\textrm{codim}\,\varphi_0(c^k)$ --- number of algebraically independent Casimirs provided by $\varphi_0(c^k)$.
\begin{table}[h!]
\begin{tabular}{c|ccc}
$N$&$\textrm{dim}\,\mathbb C_{\mathcal V}^{inv}$&$\textrm{dim}\,L$&$\textrm{codim}\,\varphi_0(c^k)$\\
\hline
1&2&2&0\\
2&5&4&1\\
3&10&8&2\\
4&17&14&3\\
5&26&22&4\\
6&37&32&5
\end{tabular}
\caption{Summary on tests of dimensions of symplectic leaf}
\label{tab:QuadraticBracketSymplecticLeafDimensions}
\end{table}

\subsection{Other examples Double Poisson Brackets}

Below we present a couple of other examples of modified double Poisson brackets on $Free_3=\mathbb C\langle x_1,x_2,x_3\rangle$. Unlike (\ref{eq:DoubleBracketForKontsevichSystem}), examples presented in this subsection are conjectural although very well tested. More examples and partial classification are in progress.
\begin{equation}
\begin{aligned}
\db{x_1,x_2}^{I}&=-x_2x_1\otimes1,&\qquad \db{x_2,x_1}^{I}&=x_1x_2\otimes1,\\
\db{x_2,x_3}^{I}&=-x_2\otimes x_3,&\qquad \db{x_3,x_2}^{I}&=x_2\otimes x_3,\\
\db{x_3,x_1}^{I}&=-1\otimes x_3x_1,&\qquad \db{x_1,x_3}^{I}&=1\otimes x_1x_3.
\end{aligned}
\label{eq:ExampleFree3I}
\end{equation}
Here all omitted brackets of generators are assumed to be zero.
\begin{equation}
\begin{aligned}
\db{x_1,x_2}^{II}&=-x_1\otimes x_2,&\qquad \db{x_2,x_1}^{II}&=x_1\otimes x_2,\\
\db{x_2,x_3}^{II}&=x_3\otimes x_2,&\qquad \db{x_3,x_2}^{II}&=-x_3\otimes x_2,\\
\db{x_3,x_1}^{II}&=x_1\otimes x_3-x_3\otimes x_1.
\end{aligned}
\label{eq:ExampleFree3II}
\end{equation}
\begin{conjecture}
Brackets (\ref{eq:ExampleFree3I}) and (\ref{eq:ExampleFree3II}) are modified double Poisson brackets on $Free_3$. Namely, corresponding biderivations satisfy (\ref{eq:ModifiedDoublePoissonJacobi}) and (\ref{eq:ModifiedDoublePoissonWeakSkewSymmetry}).
\end{conjecture}

We have tested equations (\ref{eq:ModifiedDoublePoissonJacobi}) and (\ref{eq:ModifiedDoublePoissonWeakSkewSymmetry}) for all monomials up to length 5.

\section{Modified double Poisson bi-vectors}
\label{sec:GeneralPoissonPolyvectors}

Following \cite{Crawley-Boevey'1999,Crawley-BoeveyEtingofGinzburg'2007,VandenBergh'2008} we introduce the notions of noncommutative vector fields and algebra of poly-vector fields. We slightly modify homomorphism from noncommutative poly-vectors to derivations introduced in \cite{VandenBergh'2008} to include poly-derivations with no cyclic invariance.
\begin{definition}
For a finitely generated associative algebra $\mathcal A$, let $\delta:\mathcal A\rightarrow\mathcal A\otimes\mathcal A$ be a linear map satisfying the following form of Leibnitz identity
\begin{align*}
\delta(ab)=(a\otimes1)\,\delta(b)+\delta(a)\,(1\otimes b).
\end{align*}
Then we call $\delta$ a \textbf{noncommutative vector field}. The space of all noncommutative vector fields for a given algebra $\mathcal A$ we denote by $\mathcal D_{\mathcal A}$ in what follows.
\end{definition}

In other words $\delta\in\mathcal D_{\mathcal A}$ is a derivation of $\mathcal A$ with a codomain $\mathcal A\otimes\mathcal A$ treated as an outer $\mathcal A$-bimodule. The remaining structure of inner bimodule makes $\delta$ itself an $\mathcal A$-bimodule, where the left-right action is defined s.t. for $\delta\in\mathcal D_{\mathcal A}$ and for all $a_1,a_2,x\in{\mathcal A}:$
\begin{align}
a_1\,\delta(x)\,a_2:=(1\otimes a_1)\,\delta(x)\,(a_2\otimes 1)=\delta(x)'a_2\otimes a_2\delta(x)''.
\label{eq:DAInnerBimoduleAction}
\end{align}
Let $\mathcal D\mathcal A=T_{\mathcal A}\mathcal D_{\mathcal A}$ be the tensor algebra over $\mathcal A$ generated by $\mathcal A$-bimodule (\ref{eq:DAInnerBimoduleAction}). As suggested in \cite{VandenBergh'2008} we call elements of $\mathcal D\mathcal A$ noncommutative poly-vector fields. One can equip $\mathcal D\mathcal A$ with a grading by assigning $\deg a=0$ for each $a\in \mathcal A$ and $\deg\delta=1$ for each $\delta\in\mathcal D_{\mathcal A}$. We denote multiplication in $\mathcal{DA}$ by $\star$ and $k^{th}$ homogenous component by $(\mathcal{DA})_k$. In Sweedler notations multiplication of noncommutative poly-vector fields becomes explicit
\begin{align}
&D=\delta_1\star\delta_2\star\dots\star\delta_k\in (\mathcal D\mathcal A)_k,\quad D:\mathcal A^{\otimes k}\rightarrow\mathcal A^{\otimes(k+1)},\nonumber\\
& D(a_1\otimes\dots\otimes a_k)=\delta_k(a_k)' \otimes\delta_{k-1}(a_{k-1})'\delta_k(a_k)''
\otimes\dots \otimes\delta_1(a_1)'\delta_2(a_2)'' \otimes\delta_1(a_1)''
\label{eq:PolyvectorFieldAction}
\end{align}

Unpaired components of vector fields in (\ref{eq:PolyvectorFieldAction}) provide a structure on an $\mathcal A$-bimodule, for each $b\in\mathcal A$ and $D\in(\mathcal D\mathcal A)_k$
\begin{align*}
(D\star b)(a_1\otimes a_k)=\delta_k(a_k)'b\otimes\dots\otimes\delta_1(a_1)'',\\
(b\star D)(a_1\otimes a_k)=\delta_k(a_k)'\otimes\dots\otimes b\delta_1(a_1)''.
\end{align*}
For each $\mathcal A$-bimodule one can define a notion of partial trace, so we introduce
\begin{definition} Let $P\in\mathcal{DA}$, we call ``partial trace over $\mathcal A$'' and denote as $\textrm{tr}_{\mathcal A}P$ the following equivalence class
\begin{align*}
\textrm{tr}_{\mathcal A}P=P+[\mathcal{DA}, \mathcal A],\quad\textrm{where}\quad [\mathcal{DA},\mathcal A]=\textrm{Span}\left\{Q\star a-a\star Q\;|\;a\in\mathcal A,\;Q\in\mathcal{DA}\right\},
\end{align*}
or equivalently
\begin{align*}
\textrm{tr}_{\mathcal A}:\;\mathcal{DA}\rightarrow\mathcal{DA}/[\mathcal{DA},\mathcal A].
\end{align*}
\end{definition}
\begin{definition}
Let $\mathcal D_{\mathcal A^{\otimes k}}$ be the following space of $k$-derivations $\delta:\mathcal A^{\otimes k}\rightarrow\mathcal A^{\otimes k}$ s.t. for all $a_1,\dots,a_k\in\mathcal A$ and for all $i\in\{1,\dots,k\}$
\begin{equation}
\begin{aligned}
\delta(a_1\otimes\dots\otimes \underset{\displaystyle\underset i\uparrow}{\displaystyle bc}\otimes\dots\otimes a_k)=&(\underset{k-i}{\undergroup{1\otimes\dots\otimes1}}\otimes b\otimes\underset{i}{\undergroup{1\otimes\dots\otimes 1}})\,\delta(a_1\otimes\dots\otimes\underset{\displaystyle\underset i\uparrow} c\otimes\dots\otimes a_k)+\\
&+\delta(a_1\otimes\dots\otimes\underset{\displaystyle\underset i\uparrow} b\otimes\dots\otimes a_k)\,(\underset{(1-i) \bmod k}{\undergroup{1\otimes\dots\otimes1}}\otimes c\,\otimes\underset{((i-2)\bmod k)+1}{\undergroup{1\otimes\;\;\dots\;\;\otimes1}}).
\end{aligned}
\label{eq:BasicPolyderivation}
\end{equation}
\end{definition}

\begin{proposition}
Partial trace provides a linear homomorphism $\left(\mathcal{DA}/[\mathcal{DA},\mathcal A]\right)_k\xrightarrow{\textrm{tr}_{\mathcal A}}\mathcal D_{\mathcal A^{\otimes k}}$ given by
\begin{align}
\left(\textrm{tr}_{\mathcal A}(\delta_1\star\dots\star\delta_k)\right)(a_1\otimes\dots\otimes a_k)=\delta_k(a_k)'\delta_1(a_1)''\otimes\delta_{k-1}(a_{k-1})' \delta_k(a_k)''\otimes\dots\otimes\delta_1(a_1)'\delta_2(a_2)''.
\label{eq:PartialTraceExplicit}
\end{align}
\label{prop:polyderivations}
\end{proposition}

\begin{proof}
First, recall that $\delta_i(bc)=b\delta_i(c)'\otimes\delta_i(c)'' +\delta_i(b)'\otimes\delta_i(b)''c,$ then substituting it in to the l.h.s. of (\ref{eq:BasicPolyderivation}) we get the desired result.
\end{proof}

\begin{corollary}
Let $\delta_1,\delta_2\in D_A$ be noncommutative vector fields, then $R=\textrm{tr}_{\mathcal A}(\delta_1\star\delta_2)$ is a biderivation:
\begin{align*}
R\,(ab\otimes c)=(1\otimes a) \left(R\,(b\otimes c)\right)+\left(R\,(a\otimes c)\right)(b\otimes 1),\\
R\,(a\otimes bc)=(b\otimes 1)\left(R\,(a\otimes c)\right)+\left(R\,(a\otimes b)\right)(1\otimes c),
\end{align*}
where
\begin{align*}
R(a\otimes b):=\textrm{tr}_{\mathcal A}(\delta_1\star\delta_2)(a\otimes b)=\delta_2(b)'\delta_1''(a)\otimes\delta_1(a)'\delta_2(b)''.
\end{align*}
\label{lemm:BiderivationFromBivector}
\end{corollary}
%\begin{proof}
%We prove only the first statement, the proof of the second one is essentially the same. Recall
%\begin{subequations}
%\begin{align}
%\delta_1(ab)=&\sum_i\delta_{1,i}(ab)'\otimes\delta_{1,i}(ab)',
%\label{eq:DDLemmaNCVectorFieldSweedler}\\
%\delta_1(ab)=&(a\otimes1)\,\delta_1(b)+\delta_1(a)\,(1\otimes b)\nonumber\\
%=&\sum_ia\,\delta_1(b)'\otimes\delta_1(b)'' +\sum_i\delta_1(a)'\otimes\delta_1(a)''\,b.
%\label{eq:DDLemmaNCVectorFieldLeibnitz}
%\end{align}
%\end{subequations}
%So one can break the sum in (\ref{eq:DDLemmaNCVectorFieldSweedler}) into two pieces (\ref{eq:DDLemmaNCVectorFieldLeibnitz}), as a result we get
%\begin{align*}
%\delta_1\star\delta_2\,(ab\otimes c)=&\delta_2(c)'\delta_1(ab)''\otimes\delta_1(ab)'\delta_2(c)''\\
%=&\delta_2(c)'\delta_1(b)''\otimes a\,\delta_1(b)'\delta_2(c)'' +\delta_2(c)'\delta_1(a)''\,b\otimes\delta_1(a)'\delta_2(c)''\\
%=&(1\otimes a)\left(\delta_1\star\delta_2\,(b\otimes c)\right)+\left(\delta_1\star\delta_2\,(a\otimes c)\right)(b\otimes1).
%\end{align*}
%\end{proof}

\begin{definition}
Let $P=\sum_i\delta_1^i\star\delta_2^i\in(\mathcal{DA})_2$ be a noncommutative bi-vector, we call $P$ a \textbf{Modified double Poisson bi-vector} if the induced bi-derivation $\textrm{tr}_{\mathcal A}\left(\sum_i\delta_1^i\star\delta_2^i\right):A\otimes A\rightarrow A\otimes A$ is a modified double Poisson bracket.
\end{definition}

Proposition \ref{prop:polyderivations} should be compared to Proposition 4.1.1 in \cite{VandenBergh'2008}. It is worth noting that we have used partial trace $\textrm{tr}_{\mathcal A}:\;\mathcal{DA}\rightarrow\mathcal{DA}/[\mathcal{DA},\mathcal A]$ to construct bi-derivations by the corresponding poly-vector fields as opposed to the abelianization $\mathcal{DA}\rightarrow\mathcal{DA}/[\mathcal{DA},\mathcal{DA}]$ used in \cite{VandenBergh'2008}. This allows one to take into consideration poly-derivations with no cyclic ``anti-equivariance''.

As we have shown in Sec. \ref{sec:ExamplesOfModifiedDoublePoissonBrackets} there are examples of the modified double Poisson brackets given by a biderivation which is substantially non-skew-symmetric. Below we present non-skew-symmetric double Poisson bivector which induces bracket (\ref{eq:DoubleBracketForKontsevichSystem}).

\subsection*{Poisson bivector for bracket (\ref{eq:DoubleBracketForKontsevichSystem})}
\label{sec:PoissonBivectorForKontsevichBracket}

Let $\mathcal A^+=\mathbb C\langle u,v\rangle$ be a free associative algebra with two generators. Define noncommutative vector fields $\delta_1,\delta_2,\widetilde{\delta_1},\widetilde{\delta_2}\in D_{\mathcal A}$ by their action on generators
\begin{align*}
\begin{array}{lll}
\delta_1(u)=1\otimes u,&&\delta_1(v)=1\otimes v,\\
\delta_2(u)=u\otimes1,&&\delta_2(v)=0,\\
\widetilde{\delta_1}(u)=u\otimes 1,&&\widetilde{\delta_1}(v)=v\otimes 1,\\
\widetilde{\delta_2}(u)=1\otimes u,&&\widetilde{\delta_2}(v)=0.
\end{array}
\end{align*}
\begin{proposition}
\begin{align*}
\db{\_}_K:=\textrm{tr}_{\mathcal A^+}\left(\delta_1\star\delta_2- \widetilde{\delta_2}\star\widetilde{\delta_1}\right)
\end{align*}
\end{proposition}
\begin{proof}
On generators of $\mathcal A$ formula above coincide with an example of modified double Poisson bracket $\db{\_,\_}_K$ defined in (\ref{eq:DoubleBracketForKontsevichSystem}). From Corollary \ref{lemm:BiderivationFromBivector} we conclude that they coincide for the entire domain $\mathcal A\otimes\mathcal A$.
\end{proof}

Thus $\delta_1\star\delta_2- \widetilde{\delta_2}\star\widetilde{\delta_1}\in (\mathcal{DA})_2$ provides an essential example of a modified double Poisson bivector.

\section{Brackets on representation algebras}
\label{sec:BracketsOnRepresentationAlgebras}

In \cite{Turaev'2014,MassuyeauTuraev'2015} G.~Massuyeau and V.~Turaev suggested that double Poisson brackets induce Poisson brackets on representation algebras. In particular, it was shown that for each associative algebra $\mathcal A$ and coalgebra $\mathcal M$ one can define a commutative associative algebra $\mathcal A_{\mathcal M}$ satisfying the following universal property: for any commutative algebra $\mathcal B$ consider $Hom(\mathcal M,\mathcal B)$ as an associative algebra with convolution product, then for each $s:\;\mathcal A\rightarrow Hom(\mathcal M,\mathcal B)$ there exists a unique $r:\;\mathcal A_{\mathcal M}\rightarrow\mathcal B$ s.t. the following diagram is commutative in the category of associative algebras
\begin{align*}
\xymatrix{
A\ar[r]\ar[dr]_s&Hom(\mathcal M,\mathcal A_{\mathcal M})\ar[d]^r\\
&Hom(\mathcal M,\mathcal B).
}
\end{align*}

In this section we investigate brackets on representation algebras induced by the modified double Poisson brackets. We start by a very brief review of representation algebras in Section \ref{sec:RepresentationAlgebras} followed by Section \ref{sec:ModifiedBracketsOnRepresentationAlgebras} in which we show that modified double bracket induces a Poisson bracket on the ``trace'' subalgebra of the representation algebra. The ``trace'' subalgebra acts (as a Lie algebra) on the entire representation algebra by derivations.

\subsection{Representation Algebras}
\label{sec:RepresentationAlgebras}

Throughout this section let $\mathcal A$ be an associative algebra and $\mathcal M$ be a coassociative coalgebra with comultiplication $\Delta:\mathcal M\rightarrow\mathcal M\otimes\mathcal M$ and counit $\epsilon:\mathcal M\rightarrow\mathbb C$. Since $\Delta$ is coassociative, for each $m>1$ comultiplication induces a unique map $\Delta^m:\mathcal M\rightarrow\mathcal M^{\otimes m}$. It will be useful for us to employ the following notations: $\Delta^m(\alpha)=:\alpha^1\otimes\dots\otimes\alpha^m$ whenever $\alpha\in\mathcal M$.

Assume further, that $\mathcal M$ is equipped with:
\begin{enumerate}
\item Bilinear form $\nu:\mathcal M\otimes\mathcal M\rightarrow\mathbb C$ s.t. $\forall\alpha,\beta\in\mathcal M,\;\nu(\alpha,\beta^2)\beta^1\otimes\beta^3= \nu(\beta,\alpha^2)\alpha^1\otimes\alpha^3;$
\item A ``trace'' element $\tau\in\mathcal M,$ s.t. $\forall\alpha\in\mathcal M,\;\nu(\tau,\alpha)=\epsilon(\alpha).$
\end{enumerate}
Together with comultiplication a bilinear form induces a map
\begin{align*}
\bar{\nu}:\mathcal M\otimes\mathcal M\rightarrow\mathcal M\otimes\mathcal M,\quad\bar{\nu}(\alpha,\beta)=\nu(\alpha,\beta^2) \beta^1\otimes\beta^3.
\end{align*}
We denote the image of this map in Sweedler notations as $\bar{\nu}(\alpha,\beta)=\alpha_\beta\otimes\beta^\alpha.$

The requirement for $\nu$ formulated above is equivalent to the symmetrycity of $\bar{\nu}$, namely
\begin{align*}
\nu(\alpha,\beta^2)\beta^1\otimes\beta^3= \nu(\beta,\alpha^2)\alpha^1\otimes\alpha^3
\quad\Longleftrightarrow\quad
\alpha_\beta\otimes\beta^\alpha=\beta_\alpha\otimes\alpha^\beta
\end{align*}
\begin{definition}\cite{Turaev'2014}
Let $\mathcal A_{\mathcal M}$ be a commutative algebra generated by all such elements $a_{\alpha},\;a\in A,\,\alpha\in\mathcal M,$ subject to the relations
\begin{enumerate}
\item $\forall k\in\mathbb C,\,a,b\in\mathcal A,$ and $\alpha, \beta\in\mathcal M,$
\begin{align*}
k(a_\alpha)=(ka)_\alpha=a_{k\alpha},\quad (a+b)_\alpha=a_\alpha+b_\alpha,\quad a_{\alpha+\beta}=a_{\alpha}+a_{\beta};
\end{align*}
\item $\forall a,b\in\mathcal A$ and $\alpha\in\mathcal M$,
\begin{align*}
(ab)_\alpha=a_{\alpha^1}b_{\alpha^2}.
\end{align*}
\end{enumerate}
We call $\mathcal A_{\mathcal M}$ a \textbf{representation algebra} of
$\mathcal A$ in $\mathcal M$.
\end{definition}

\begin{lemma}\cite{Turaev'2014}
\begin{subequations}
\begin{align}
\alpha_\beta\otimes(\beta^\alpha)^1\otimes(\beta^\alpha)^2=& \alpha_{(\beta^1)}\otimes(\beta^1)^\alpha\otimes\beta^2,\\
(\alpha_\beta)^1\otimes(\alpha_\beta)^2\otimes\beta^\alpha=& \beta^1\otimes\alpha_{(\beta^2)}\otimes(\beta^2)^\alpha.
\end{align}
\end{subequations}
\label{lemm:CoalgebraSymmetricBilinearFormIdentities}
\end{lemma}

%\begin{example}
%Let $\mathcal M=(Mat(N,\mathbb C))^*$, denote $\tau_{ij}:\;\mathcal M\rightarrow\mathbb C,\;\tau_{ij}(m)=m_{ij}$. It acquires the structure of coassociative coalgebra with comultiplication $\Delta(\tau_{i,j})=\sum_k \tau_{ik}\otimes \tau_{kj}$. The role of the counit is played by the evaluation at identity matrix $\epsilon:\;\mathcal M\rightarrow\mathbb C,\;\epsilon(m):=m(\textrm{Id}_N).$
%\end{example}

\subsection{Modified brackets on representation algebras}
\label{sec:ModifiedBracketsOnRepresentationAlgebras}

\begin{lemma}
Modified double Poisson bracket $\db{,}$ induces a biderivation on representation algebra $A_{\mathcal M}$
\begin{align*}
\{,\}^{\mathcal M}:\mathcal A_{\mathcal M}\otimes\mathcal A_{\mathcal M}\rightarrow\mathcal A_{\mathcal M}\qquad \{a_\alpha,b_\beta\}^{\mathcal M}:=\db{a,b}'_{\alpha_\beta} \db{a,b}''_{\beta^\alpha} =\nu(\alpha,\beta^{2})\,\db{a,b}'_{\beta^1}\,\db{a,b}''_{\beta^2}
\end{align*}
\end{lemma}
\begin{proof}
Define $\{,\}^{\mathcal M}$ on generators of $\mathcal A_{\mathcal M}$ as above and then extend to arbitrary pairs of monomials by Leibnitz identity. We have to show that defining relations of $\mathcal A_{\mathcal M}$ are annihilated by $\{,\}^{\mathcal M}$. For all $a,b,c\in\mathcal A$ and $\alpha,\beta\in\mathcal A_{\mathcal M}$ we get
\begin{align*}
\{(ab)_\alpha,c_\beta\}^{\mathcal M} =&\db{ab,c}'_{\alpha_\beta}\db{ab,c}''_{\beta^\alpha}\\ =&\db{b,c}'_{\alpha_\beta}(a\db{b,c}'')_{\beta^\alpha} +(\db{a,c}'b)_{\alpha_\beta}\db{a,c}''_{\beta^\alpha}\\ =&\db{b,c}'_{\alpha_\beta}a_{(\beta^\alpha)^1} \db{b,c}''_{(\beta^\alpha)^2}+\db{a,c}'_{(\alpha_\beta)^1} b_{(\alpha_\beta)^2}\db{a,c}''_{\beta^\alpha}\\
\textrm{(by Lemma \ref{lemm:CoalgebraSymmetricBilinearFormIdentities})} =&a_{\alpha^1}\db{b,c}'_{(\alpha^2)_\beta} \db{b,c}''_{\beta^{(\alpha^2)}}+b_{\alpha^2} \db{a,c}'_{(\alpha^1)_\beta}\db{a,c}''_{\beta^{(\alpha^1)}}\\
=&a_{\alpha^1}\{b_{\alpha^2},c_\beta\}^{\mathcal M} +b_{\alpha^2}\{a_{\alpha^1},c_\beta\}^{\mathcal M}\\
=&\{a_{\alpha^1}b_{\alpha^2},c_\beta\}^{\mathcal M}.
\end{align*}
Similar computation shows that $\{a_{\alpha},(bc)_{\beta}\}^{\mathcal M}=\{a_{\alpha},b_{\beta^1}c_{\beta^2}\}^{\mathcal M}.$
\end{proof}

\begin{lemma}
For all $\alpha\in\mathcal M$ and $a,b\in\mathcal A$:
\begin{align*}
\{a_{\tau},b_{\alpha}\}^{\mathcal M}=\left(\{a,b\}\right)_{\alpha}.
\end{align*}
\label{lemma:RepresentationAlgebraTraceBracket}
\end{lemma}
\begin{proof}
First, we use the fact that $\tau$ is conjugate to counit to show that for all $\alpha\in\mathcal M$
\begin{align}
\tau_\alpha\otimes\alpha^\tau:=\nu(\tau,\alpha^2) \alpha^1\otimes\alpha^2=\epsilon(\alpha^2)\alpha^1\otimes\alpha^3= \alpha^1\otimes\alpha^2.
\label{eq:RepresentationAlgebraTraceBracket}
\end{align}
Now,
\begin{align*}
\{a_{\tau},b_{\alpha}\}^{\mathcal M}=&\;\db{a,b}'_{\tau_{\alpha}} \db{a,b}''_{\alpha^\tau}\\ \stackrel{\mathclap{\tiny \mbox{(\ref{eq:RepresentationAlgebraTraceBracket})}}}{=}&\; \db{a,b}'_{\alpha^1}\db{a,b}''_{\alpha^2} =\left(\db{a,b}'\db{a,b}''\right)_{\alpha}\\
=&\left(\{a,b\}\right)_{\alpha}.
\end{align*}
\end{proof}
\begin{proposition}
The following restriction
\begin{align*}
\{,\}^{\mathcal M}:\;\mathcal A_{\tau}\otimes\mathcal A_{\mathcal M}\rightarrow\mathcal A_{\mathcal M}
\end{align*}
satisfies Jacobi identity for left Loday bracket, namely for all $\alpha\in\mathcal M$ and $a,b,c\in\mathcal A$:
\begin{align*}
\{a_\tau,\{b_\tau,c_\alpha\}^{\mathcal M}\}^{\mathcal M}-\{b_\tau,\{a_\tau,c_\alpha\}^{\mathcal M}\}^{\mathcal M} =\{\{a_\tau,b_\tau\}^{\mathcal M},c_\alpha\}^{\mathcal M}.
\end{align*}
\end{proposition}
\begin{proof}
Compose (\ref{eq:H0LeftLodayJacobi}) with Lemma \ref{lemma:RepresentationAlgebraTraceBracket}
\end{proof}

\begin{corollary}
Subspace $\mathcal A_\tau\subset\mathcal A_{\mathcal M}$ is a Lie algebra w.r.t. $\{,\}^{\mathcal M}$.
\end{corollary}

\begin{corollary}
Let $\mathbb C[\mathcal A_\tau]\subset\mathcal A_{\mathcal M}$ be a commutative algebra generated by $\mathcal A_\tau$, then the following restriction of $\{,\}^{\mathcal M}$
\begin{align*}
\{,\}^\tau:\;\mathbb C[\mathcal A_\tau]\otimes\mathbb C[\mathcal A_\tau]\rightarrow\mathbb C[\mathcal A_\tau]
\end{align*}
is a Poisson bracket.
\end{corollary}

\section*{Acknowledgements}

I would like to thank S.~Chung, Prof.~M.~Kontsevich, Prof.~G.~Powell, Prof.~M.~Shapiro, Prof.~V.~Turaev, and Prof.~C.~Weibel for useful remarks. I am especially grateful to Prof.~V.~Retakh and Prof.~V.~Roubtsov for formulating the problem, guidance and numerous fruitful discussions. I would also like to thank IHES for hospitality during my visit in January 2016, where part of this work was done.

\section{Conclusion and discussion}

Throughout this note we have shown that one can generalize fundamental results of \cite{VandenBergh'2008} with no assumption of skew-symmetricity of the double bracket
\begin{align*}
\db{a,b}'\otimes\db{a,b}''=-\db{b,a}''\otimes\db{b,a}'.
\end{align*}
We have presented an essential example of non-skew-symmetric modified double Poisson bracket and have constructed the corresponding modified double Poisson bi-vector.

We conclude by remark on Jacobi identity for the double Poisson bracket

\subsection*{Jacobi identity beyond triple derivations}

Jacobi identity for the Loday bracket induced by the modified double bracket can be presented in the following form
\begin{align}
0=&\{H_1,\{H_2,x\}\}-\{H_2,\{H_1,x\}\}-\{\{H_1,H_2\},x\}=\mu(D_1+D_2),
\label{eq:JacobiDecomposition}
\end{align}
where
\begin{align*}
D_1:=&R_{12}R_{23}-R_{23}R_{13}-R_{13}R_{12},\\
D_2:=&\sigma_{12}\left(R_{13}R_{23}-R_{21}R_{13}-R_{23}R_{12}\right),
\end{align*}
\begin{align*}
R_{m,n}:&\quad\mathcal A^{\otimes k}\rightarrow\mathcal A^{\otimes k},\quad R_{m,n}(a_1\otimes\dots\otimes a_k)=a_1\otimes\dots\otimes\underset{\displaystyle\underset i\uparrow}{\db{a_i,a_j}'}\otimes\dots\otimes\underset{\displaystyle\underset j\uparrow}{\db{a_i,a_j}''}\otimes\dots\otimes a_k.
\end{align*}

Definition of the usual double Poisson bracket \cite{VandenBergh'2008} included strong requirement of the skew-symmetry $R_{m,n}=-\sigma_{(m,n)}R_{n,m}\sigma_{(m,n)}$ and the so-called double Jacobi identity (analogue of Yang-Baxter equation for double bracket). Together they guarantee that $D_1$ and $D_2$ are triple derivations and vanish separately. This fact was heavily used in various classification problems of double Poisson brackets \cite{OdesskiiRubtsovSokolov'2013,Powell'2016}. However, this is no longer the case for modified double Poisson brackets (in particular, it fails for (\ref{eq:DoubleBracketForKontsevichSystem})). Indeed, consider a defect of $D_1$ being a derivation w.r.t. to the first argument
\begin{equation}
\begin{aligned}
D_1(a_1a_2\otimes b\otimes c)-(1\otimes a_1\otimes1)D_1(a_2\otimes b\otimes c)-D_1(a_1\otimes b\otimes c)(a_2\otimes1\otimes1)=\\
-\db{a_2,c}'\otimes\db{b,a_1}'\otimes\db{b,a_1}''\db{a_2,c}'' -\db{a_2,c}'\otimes\db{a_1,b}''\otimes\db{a_1,b}'\db{a_2,c}''.
\end{aligned}
\label{eq:YangBaxtexDefect}
\end{equation}
Beyond the skew-symmetric case, the right hand side of (\ref{eq:YangBaxtexDefect}) doesn't vanish in general. Instead, only the combination composed with multiplication map (\ref{eq:JacobiDecomposition}) vanishes.

The latter makes classification problem for modified double Poisson brackets rather challenging. On the other hand, Definition \ref{def:ModifiedDoubleBracket} is completely explicit for finitely generated algebra and it is enough to define its' action on generators. It is an interesting topic of further research to formulate an explicit condition (namely, that can be tested on generators only) which can replace (\ref{eq:ModifiedDoublePoissonJacobi}) and (\ref{eq:ModifiedDoublePoissonWeakSkewSymmetry}) in Definition \ref{def:ModifiedDoubleBracket}.

\appendix
\section{Spectral curve for Kontsevich system}

In this appendix we describe an application of modified double Poisson brackets to the noncommutative Integrable System suggested in \cite{Kontsevich'2011}. We show that it induces Liouville Integrable System on the moduli space of representations of the underlying associative algebra for small $N$. Moreover, using Proposition \ref{prop:H0BracketNXN} we show that the corresponding Hamilton flows extend to the derivations of the entire coordinate space of representations. Finally, from noncommutative Lax pair suggested in \cite{EfimovskayaWolf'2012} we construct the corresponding spectral curve and compute its genus.

\subsection{Kontsevich system}
Let $A=\mathbb C\langle u^{\pm1},v^{\pm1}\rangle$ denote the group algebra of the free group with two generators. Consider derivation $\frac{\textrm d}{\textrm dt}:A\rightarrow A$ defined by
\begin{align}
\left\{\begin{array}{l}
\dfrac{\textrm du}{\textrm dt}=uv-uv^{-1}-v^{-1},\\[10pt]
\dfrac{\textrm dv}{\textrm dt}=-vu+vu^{-1}+u^{-1}.
\end{array}\right.
\label{eq:KontsevihSystem}
\end{align}
Recall the modified double Poisson bracket (\ref{eq:DoubleBracketForKontsevichSystem})
\begin{align}
\db{u,v}_K=-vu\otimes1,\qquad
\db{v,u}_K=uv\otimes 1,\qquad \db{u,u}_K=\db{v,v}_K=0.
\label{eq:DoubleBracketForKontsevichSystemAppendix}
\end{align}
Denote the induced $H_0$-Poisson structure as
\begin{align}
\{,\}_K:A\otimes A\rightarrow A;\quad\forall a,b\in A,\;\{a,b\}_K=\mu(\db{a,b}_K).
\label{eq:H0BracketForKontsevichSystem}
\end{align}
Bracket (\ref{eq:DoubleBracketForKontsevichSystemAppendix}) was suggested in \cite{Arthamonov'2015} by the author to show the integrability of (\ref{eq:KontsevihSystem}). We have the following list of properties:
\begin{itemize}
\item Derivation (\ref{eq:KontsevihSystem}) is a generalized Hamilton flow w.r.t. bracket (\ref{eq:H0BracketForKontsevichSystem}), namely
    \begin{align*}
    \forall x\in A,\quad \frac{\textrm dx}{\textrm dt}=\{h,x\}_K,\qquad\textrm{where}\qquad
    h=u+v+u^{-1}+v^{-1}+u^{-1}v^{-1}.
    \end{align*}
\item There exists an infinite family of commuting flows, for all $k,j\in\mathbb N$
    \begin{align*}
    \frac{\textrm d}{\textrm dt_k}:A\rightarrow A,\quad\frac{\textrm d}{\textrm dt_k}(x):=\{h^k,x\};\qquad \left[\frac{\textrm d}{\textrm dt_k},\frac{\textrm d}{\textrm dt_j}\right]=0.
    \end{align*}
\item Group commutator $c=uvu^{-1}v^{-1}$ generates the subalgebra of right Casimirs of bracket (\ref{eq:H0BracketForKontsevichSystem})
    \begin{align*}
    \forall x\in A/[A,A],\;\forall C\in\mathbb C[c]:\quad \{x,C\}_K=0.
    \end{align*}
\end{itemize}

\subsection{Basic example. Matrix representations for $N=2$}

In this subsection we show that Kontsevich system induces a conventional Integrable System in the Liouville sense on the moduli space of 2-dimensional representations. Moreover, we show that Hamilton flows extend to the entire coordinate space of representations.

Define an 8-dimensional manifold $\mathcal M\subset R^8$ with coordinates $u_{11},u_{12},u_{21},u_{22},\-v_{11},v_{12},v_{21},v_{22}$ s.t. $u_{11}u_{22}-u_{12}u_{21}\neq0$ and $v_{11}v_{22}-v_{12}v_{21}\neq0$. Let
\begin{align}
\varphi(u)=\left(\begin{array}{cc}
u_{11}&u_{12}\\
u_{21}&u_{22}
\end{array}\right),
\qquad
\varphi(v)=\left(\begin{array}{cc}
v_{11}&v_{12}\\
v_{21}&v_{22}
\end{array}\right).
\label{eq:2X2RepresentationDefinition}
\end{align}
Then $\varphi:\,\mathcal M\times A\rightarrow Mat(2,\mathbb C)$ provides a representation of $A$ for each given $m\in\mathcal M$.

Algebra $\mathbb C[\mathcal M^{inv}]$ of all $GL(2,\mathbb C)$-invariant functions on $\mathcal M$ is generated by $\varphi_0\left(\{u,v,u^2,uv,v^2\}\right)$.
\begin{equation}
\begin{aligned}
t_1:=&\textrm{Tr}\,\varphi(u)=u_{11}+u_{22}\\
t_2:=&\textrm{Tr}\,\varphi(v)=v_{11}+v_{22}\\
t_3:=&\textrm{Tr}\,\varphi(uu)=u_{11}^2+2u_{12}u_{21}+u_{22}^2\\
t_4:=&\textrm{Tr}\,\varphi(vv)=u_{11}v_{11}+u_{21}v_{12}+u_{12}v_{21}+u_{22}v_{22}\\
t_5:=&\textrm{Tr}\,\varphi(vv)=v_{11}^2+2v_{12}v_{21}+v_{22}^2
\end{aligned}
\label{eq:tVariablesDefinition}
\end{equation}

In terms of variables $t$ defined in (\ref{eq:tVariablesDefinition}) we get the following brackets
\begin{equation}
\begin{aligned}
\{t_1,t_2\}^{inv}=&-t_4\\
\{t_1,t_3\}^{inv}=&0\\
\{t_1,t_4\}^{inv}=&\frac12\left(t_1^2t_2-t_2t_3-2t_1t_4\right)\\
\{t_1,t_5\}^{inv}=&t_1t_2^2-2t_2t_4-t_1t_5\\
\{t_2,t_3\}^{inv}=&-t_1^2t_2+t_2t_3+2t_1t_4\\
\{t_2,t_4\}^{inv}=&\frac12\left(-t_1t_2^2+2t_2t_4+t_1t_5\right)\\
\{t_2,t_5\}^{inv}=&0\\
\{t_3,t_4\}^{inv}=&t_1^3t_2-t_1t_2t_3-t_1^2t_4-t_3t_4\\
\{t_3,t_5\}^{inv}=&2\left(t_1^2t_2^2-2t_1t_2t_4-t_3t_5\right)\\
\{t_4,t_5\}^{inv}=&t_1t_2^3-t_2^2t_4-t_1t_2t_5-t_4t_5
\end{aligned}
\label{eq:KontsevichBracketN2t}
\end{equation}

The subalgebra of Casimir functions is generated by a single element $\textrm{Tr}\,\varphi(c)$ and one can check independently that symplectic leaf of (\ref{eq:KontsevichBracketN2t}) has dimension 4 (see Table \ref{tab:QuadraticBracketSymplecticLeafDimensions}). The two Hamilton functions $H_1=\textrm{Tr}\,\varphi(h)$ and $H_2=\textrm{Tr}\,\varphi(h^2)$ are algebraically independent. Next, $\{H_1,H_2\}^{inv}=0$ by the fact that $\{h^m,h^n\}\equiv 0\bmod[A,A]$ and Corollary \ref{cor:DerivationHomomorphism}. Thus, system (\ref{eq:KontsevihSystem}) induces an Integrable System in the Liouville sense on $\mathcal M^{inv}$.

By Proposition \ref{prop:H0BracketNXN} the full coordinate space of representations $\mathbb C[\mathcal M]$ is a Lie module of $\mathbb C[\mathcal M^{inv}]$ with a Lie algebra structure on $C[\mathcal M^{inv}]$ given by the Poisson bracket (\ref{eq:KontsevichBracketN2t}). In particular, this implies that $H_1$ and $H_2$ generate commuting flows on the full $\mathbb C[\mathcal M]$ which preserve $\mathbb C[\mathcal M^{inv}]$.

\subsection{Spectral Curve}

Although $\mathbb C[h]\rightarrow A/[A,A]$ provides a subspace of hamiltonians in involution it doesn't span the maximal commuting Lie subalgebra w.r.t. to the Lie bracket induced by (\ref{eq:H0BracketForKontsevichSystem}) (for example see eq. (26) on p. 1237 of \cite{Arthamonov'2015}). In \cite{EfimovskayaWolf'2012} O.~Efimovskaya and T.~Wolf suggested a noncommutative Lax pair with a spectral parameter for system (\ref{eq:KontsevihSystem}) and conjectured that it will provide all "trace integrals", elements of $A$ invariant under (\ref{eq:KontsevihSystem}) modulo $[A,A]$
\begin{align*}
\frac{\textrm d}{\textrm dt}L=[L,M],
\end{align*}
where
\begin{align}
L=\left(\begin{array}{cc}
v^{-1}+u&\lambda v+v^{-1}u^{-1}+u^{-1}+1\\
v^{-1}+\frac1\lambda u&v+v^{-1}u^{-1}+u^{-1}+\frac1\lambda
\end{array}\right),
\qquad
M=\left(\begin{array}{cc}
v^{-1}-v+u&\lambda v\\
v^{-1}&u
\end{array}\right).
\label{eq:LaxPair}
\end{align}
Denote the coefficients of the ``noncommutative spectral curve'' as $\textrm{Tr}\,L(\lambda)^k=:\sum_{j=-k}^kH_{k,j}\lambda^j$. We have checked that for all $k,m\leqslant5$ and arbitrary $j,l$
\begin{align*}
\{H_{k,j},H_{m,l}\}_K\equiv0\bmod [A,A].
\end{align*}
The latter combined with results of O.~Efimovskaya and T.~Wolf suggests the following
\begin{conjecture}
Image of $\textrm{Span}(H_{k,l})$ in $A/[A,A]$ under natural projection is a maximal commutative Lie subalgebra of $A/[A,A]$ w.r.t. to the Lie bracket induced by (\ref{eq:H0BracketForKontsevichSystem}).
\label{conj:MaximalCommutativeLieSubalgebra}
\end{conjecture}

For each $N$ and each general point in $\mathcal A_N$, Lax matrix (\ref{eq:LaxPair}) gives rise to an algebraic curve in $\lambda$ and $\nu$
\begin{align}
\det(\varphi(L(\lambda))-\nu)=0.
\label{eq:KontsevichSpectralCurve}
\end{align}
The coefficients of the curve belong to the subalgebra $\mathbb C[\varphi_0(H_{k,l})]\subset\mathcal A_N^{inv}$ of invariant polynomials generated by $\varphi_0(H_{k,l})$. Assuming Conjecture \ref{conj:MaximalCommutativeLieSubalgebra}, by Corollary \ref{cor:DerivationHomomorphism} we get
\begin{corollary}
Subalgebra $\mathbb C[\varphi_0(H_{k,l})]$ generated by the coefficients of the spectral curve (\ref{eq:KontsevichSpectralCurve}) is Poisson commutative with respect to the induced bracket $\{,\}^{inv}$.
\end{corollary}

Below we present a Newton graphs of the curve (\ref{eq:KontsevichSpectralCurve}) for small $N$. For each $N$ the spectral curve appears to be highly singular, we have computed its arithmetic genus for general point in $\mathcal A_N$ for $N\leqslant4$. Table \ref{tab:KontsevichSpectralCurveGenus} suggests that the genus is equal to $N^2$.

\begin{figure}[h!]
\begin{subfigure}[b]{0.45\linewidth}
\includegraphics[width=\linewidth]{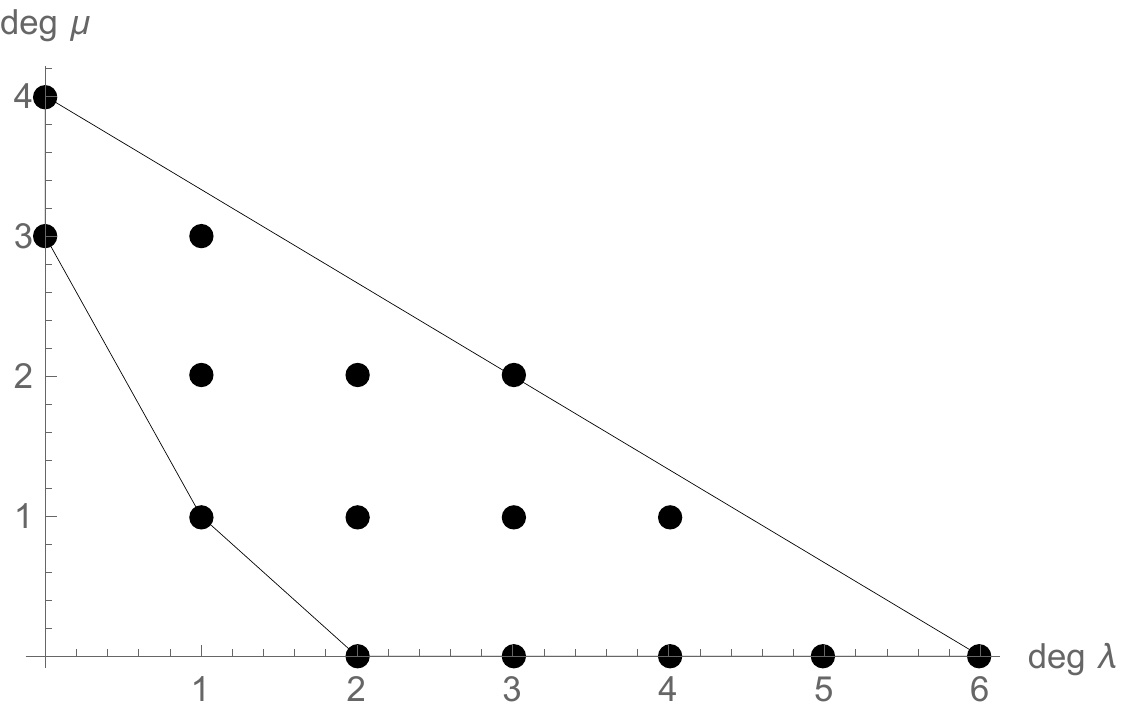}
\caption{$N=2$}
\end{subfigure}
\hspace{0.04\linewidth}
\begin{subfigure}[b]{0.45\linewidth}
\includegraphics[width=\linewidth]{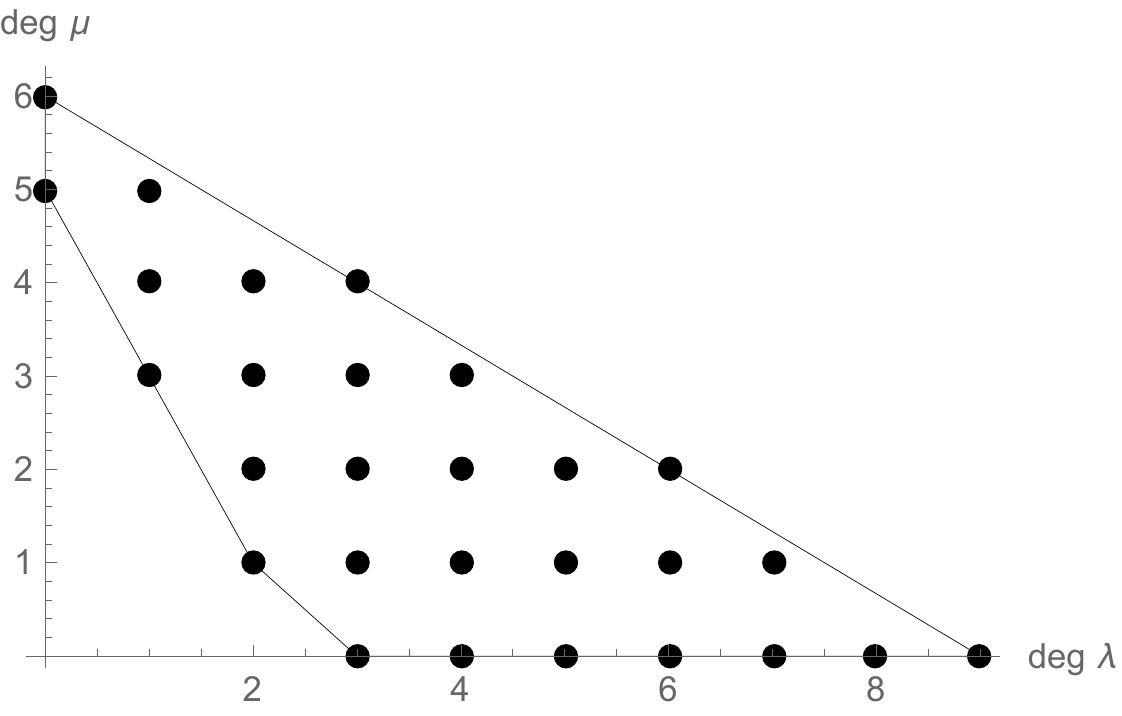}
\caption{$N=3$}
\end{subfigure}\\[10pt]
\begin{subfigure}[b]{0.45\linewidth}
\includegraphics[width=\linewidth]{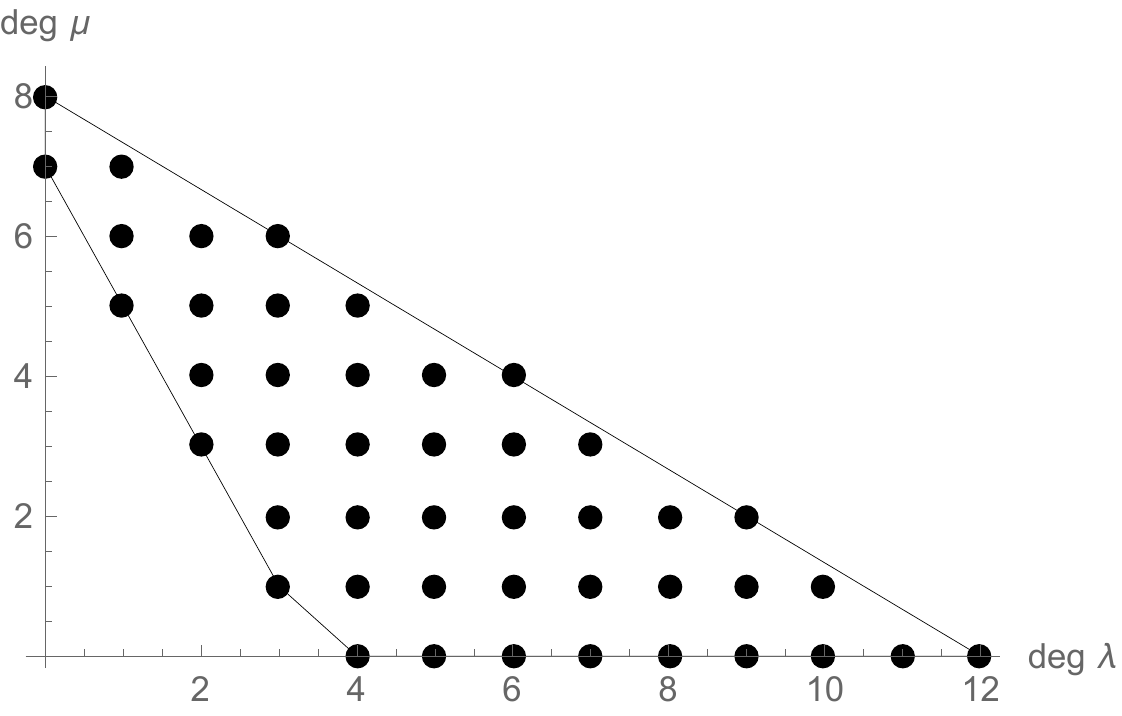}
\caption{$N=4$}
\end{subfigure}
\hspace{0.04\linewidth}
\begin{subfigure}[b]{0.45\linewidth}
\includegraphics[width=\linewidth]{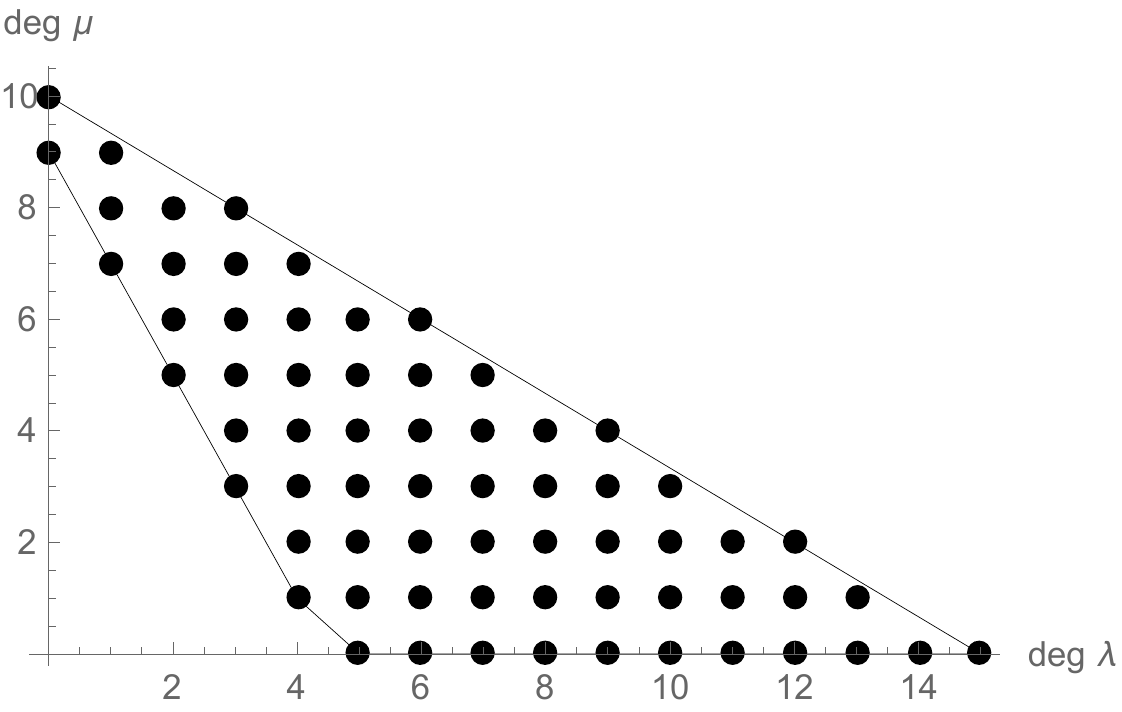}
\caption{$N=5$}
\end{subfigure}
\caption{Newton graphs of the Spectral curve}
\label{fig:NewtonGraphsForKontsevichSpectralCurve}
\end{figure}

\begin{table}[h!]
\begin{tabular}{c|cccccc}
$N$&genus&degree\\
\hline
2&4&6\\
3&9&9\\
4&16&12
\end{tabular}
\vspace{0.15cm}
\caption{Genus of the Spectral curve}
\label{tab:KontsevichSpectralCurveGenus}
\end{table}

\bibliographystyle{alpha}
\bibliography{references}

\end{document}